\documentclass[12pt, a4paper]{amsart}
\usepackage{amsmath}
\usepackage{geometry,amsthm,graphics,tabularx,amssymb,shapepar}
\usepackage{amscd}
\usepackage{graphicx}
\usepackage{xypic}

\newcommand{\Fre}{{Fr\'{e}chet \,}} 

\newcommand{\BC}{{\mathbb {C}}}

\newcommand{\BR}{{\mathbb {R}}}

\newcommand{\BZ}{{\mathbb {Z}}}

\newcommand{\CS}{{\mathcal {S}}}
\newcommand{\CT}{{\mathcal {T}}}

\newcommand{\GL}{{\mathrm{GL}}}

\newcommand{\ind}{{\mathrm{ind}}}

\newcommand{\Sym}{{\mathrm{Sym}}}

\newcommand{\oH}{\operatorname{H}}

\newcommand{\SE}{\mathsf{E}}

\newcommand{\abs}[1]{\lvert#1\rvert}

\newcommand{\be}{\begin {equation}}
\newcommand{\ee}{\end {equation}}
\newcommand{\bee}{\begin {equation*}}
\newcommand{\eee}{\end {equation*}}

\theoremstyle{Theorem}

\newtheorem{thm}{Theorem}[section]

\newtheorem{prpt}[thm]{Proposition}
\newtheorem{dfnt}[thm]{Definition}

\theoremstyle{Theorem}
\newtheorem{lem}{Lemma}[section]

\theoremstyle{Theorem}

\theoremstyle{Definition}
\newtheorem{dfn}{Definition}[section]
\newtheorem{prpd}[dfn]{Proposition}

\theoremstyle{remark}

\theoremstyle{remark}

\begin{document}

\title[Homological finiteness]{Homological finiteness of representations of almost linear Nash groups}

\author[Y. Bao]{Yixin Bao}
\address{School of Science\\
Harbin Institute of Technology\\
Shenzhen, 518055, China}
\email{mabaoyixin1984@163.com}

\author[Y. Chen]{Yangyang Chen}
\address{School of Science\\
Jiangnan University\\
Wuxi, 214122, China}
\email{8202007345@jiangnan.edu.cn}


\maketitle

\begin{abstract}
Let $G$ be an almost linear Nash group, namely, a Nash group that admits a Nash homomorphism with finite kernel to some
$\GL_k(\mathbb R)$.  A smooth \Fre representation $V$ with moderate growth of $G$ is called homologically finite if the Schwartz homology $\oH_{i}^{\CS}(G;V)$ is finite dimensional for every $i\in\BZ$. We show that the space of Schwartz sections 
$\Gamma^{\varsigma}(X,\SE)$ of a tempered $G$-vector bundle $(X,\SE)$ is homologically finite as a representation of $G$, under some mild assumptions.
\end{abstract}

\setcounter{tocdepth}{2}

\section{Introduction}

Let $G$ be an almost linear Nash group, namely, a Nash group that admits a Nash homomorphism $G\rightarrow \GL_k(\mathbb R)$ with finite kernel, for some $k\geq 0$. For details about the structure and basic properties of these groups, we refer the reader to \cite{Su}. Denote by $\CS\mathrm{mod}_G$ the category of smooth \Fre representations of $G$ of moderate growth. Recall that a representation $V$ of $G$ is said to be of moderate growth, if for every continuous seminorm 
$|\cdot|_{\mu}$ on $V$, there is a positive Nash function $f$ on $G$ and a continuous seminorm $|\cdot|_{\nu}$ on $V$ such that
\[
|g.v|_{\mu}\leq f(g)|v|_{\nu}, \quad\text{for all}\quad g\in G, v\in V.
\]
Here and as usual, we do not distinguish a representation with its underlying space.
For every representation $V$ in the category $\CS\mathrm{mod}_G$, the $i$th ($i\in\BZ$) Schwartz homology 
$\oH^{\CS}_{i}(G;V)$ of $V$, which is naturally a locally convex topological vector space, was defined and studied by the second author and Sun in \cite{CS}. It turns out that the Schwartz homologies have many nice features. On the one hand, they coincide with the smooth homologies as defined by Blanc and Wigner \cite{BW}, see \cite[Theorem 1.8]{CS}. This implies that the Schwartz homologies share the properties of the usual smooth homologies. On the other hand, the powerful tool, Shapiro's lemma also holds in the setting of Schwartz homology, now for Schwartz induced representations, which can be taken as a generalization of the usual compactly supported smooth induction, see \cite[Lemma 3.1]{LS}.

\begin{thm}[\cite{CS},Theorem 1.11]\label{shapiro}
Let $H$ be a Nash subgroup of an almost linear Nash group $G$, and let $V_0$ be a representation in $\CS\mathrm{mod}_H$.  Then there is an identification
\[
\oH_{i}^\CS(G;  \ind_H^G V_0)=\oH_{i}^\CS(H; V_0\otimes\delta_{G/H})
\]
of topological vector spaces, for every $i\in \mathbb Z$.
\end{thm}

Here $\ind_H^G$ denotes the Schwartz induction as defined by du Cloux, see \cite[Section 2]{Fd}. $\delta_G$ denotes the modular character of the group $G$ and $\delta_{G/H}:=(\delta_{G})|_H\cdot\delta_{H}^{-1}$. It should be pointed out that the Schwartz homology $\oH_i^{\CS}(G;V)$ as a topological vector space, is not necessary Hausdorff. For applications to representation theory, it is important to show that $\oH_i^{\CS}(G;V)$ is Hausdorff, at least in some cases we are interested in. In this article, we study the Hausdorffness property of the Schwartz homologies of representations of almost linear Nash groups.

\begin{dfnt}
A representation $V$ in the category $\CS\mathrm{mod}_G$ is said to be homologically separated if the Schwartz homology 
$\oH_i^{\CS}(G;V)$ is separated (Hausdorff) for every $i\in\BZ$. $V$ is said to be homologically finite if the Schwartz homology 
$\oH_i^{\CS}(G;V)$ is finite dimensional for every $i\in\BZ$.
\end{dfnt}

For the notion of homologically finite representations in the Lie algebra homology setting, see \cite[Definition 3.1.2]{AGKL}.
We study the Schwartz homologies $\oH^{\CS}_{\bullet}(G;V)$ of $V$ at all degrees together instead of a single one. Moreover, the stronger property finiteness is much easier to deal with than the Hausdorffness.

\begin{prpt}\label{fis}
Every homologically finite representation is automatically homologically separated.
\end{prpt}

This result is known to experts, see for example \cite[Lemma 3.4]{BoW} and \cite[Proposition 6]{CW}.
By \cite[Theorem 7.7]{CS}, every finite dimensional representation (with moderate growth) is homologically finite. In this article, we show that a special class of representations, which are constructed as Schwartz sections of tempered $G$-vector bundles, are homologically finite, thus also homologically separated. 

Let $X$ be a $G$-Nash manifold, namely, a Nash manifold that carries a Nash action $G\times X\rightarrow X$. Let $\SE$ be a tempered $G$-vector bundle over $X$, namely, a tempered vector bundle over $X$ together with a tempered bundle action
$G\times \mathsf E\rightarrow \mathsf E$. Then the space of Schwartz sections $\Gamma^{\varsigma}(X,\mathsf E)$ of the tempered vector bundle $\SE$ over $X$ carries a natural action of $G$, which as a representation of $G$, lies in the category 
$\CS\mathrm{mod}_G$. We will recall the notions of tempered vector bundles and Schwartz sections in the next section.

For every $x\in X$, let $G_x$ denote its stabilizer in $G$, and let $\mathsf E_x$ denote the fibre of $\mathsf E$ at $x$, which is a representation in $\CS\mathrm{mod}_{G_x}$.
Write
\[
  \mathrm{N}_{x}:=\frac{\mathrm T_x (X)}{\mathrm T_x(G.x)}\otimes_{\mathbb R} \mathbb C\qquad (\mathrm T_x\textrm{ stands for the tangent space})
\]
for the complexified  normal space, and write
\[
\mathrm{N}_{x}^*:=\textrm{the dual space of $\mathrm{N}_{x}$},
\]
which is the complexified  conormal space. They are both representations in $\CS\mathrm{mod}_{G_x}$. 
Now we can state the main results of this article.

\begin{thm}\label{main}
Let $(X,\SE)$ be a tempered $G$-vector bundle such that $G$ acts on $X$ with finitely many orbits. Assume that for every $x\in X$, the following two conditions are satisfied:
\begin{itemize}
\item $\mathsf E_x\otimes\Sym^k(\mathrm{N}_{x}^*)\otimes\delta_{G/G_x}$ is homologically finite as a representation of $G_x$, for every integer $k\geq0$, 
\item for every $i\in\BZ$, $\oH_i^{\CS}(G_x;\SE_x\otimes\Sym^k(\mathrm{N}_{x}^*)\otimes\delta_{G/G_x})$ vanishes for sufficiently large $k$,
\end{itemize}
where $\Sym^k$ indicates the $k$th symmetric power. Then $\Gamma^{\varsigma}(X,\mathsf E)$ is homologically finite as a representation of $G$.
\end{thm}

Theorem \ref{main} has the following consequence for finite rank vector bundles.

\begin{thm}\label{main2}
Let $(X,\SE)$ be a tempered $G$-vector bundle such that $G$ acts on $X$ with finitely many orbits. Assume that all the fibres of $\SE$ are finite dimensional and for every $x\in X$, the trivial representation of $G_x$ does not occur as a subquotient of
\[
\SE_x\otimes\Sym^k(\mathrm{N}_{x}^*)\otimes\delta_{G/G_x},
\]
for sufficiently large $k$. Then $\Gamma^{\varsigma}(X,\mathsf E)$ is homologically finite as a representation of $G$.
\end{thm}

By Proposition \ref{fis} and Theorem \ref{main2}, we know that the space of Schwartz sections $\Gamma^{\varsigma}(X,\mathsf E)$ of a tempered $G$-vector bundle $(X,\SE)$ is homologically separated, under some mild assumptions. Many important representations can be realized as Schwartz section space of a certain tempered $G$-vector bundle, see for example \cite[Example 1.17]{CS}. Thus we have shown that many representations are homologically separated.

This article is arranged as follows: In Section \ref{prel}, we recall the notions of tempered vector bundles, tempered bundle maps and Schwartz sections of tempered vector bundles. Then in the last section, we prove the main results of this article. The main tools to prove Theorems \ref{main} and \ref{main2} were already established by the second author and Sun in \cite{CS}.

{\bf Acknowledgements}
The authors would like to thank Binyong Sun for suggesting the definition of homologically finite representations.

\section{Tempered vector bundles and Schwartz sections}\label{prel}

In this section, we recall briefly the notions of tempered vector bundles and Schwartz sections.
For more details, we refer the reader to \cite[Sections 2 and 6]{CS}.

\subsection{Tempered vector bundles}

For basic knowledge on Nash manifolds and Nash maps, we refer the reader to \cite{Sh,Sh2}.
Let $X$ be a Nash manifold and $E_1$ and $E_2$ be two \Fre spaces. A map $\phi: X\times E_1\rightarrow E_2$ is called a liner family if the map $\phi(x, \,\cdot\,): E_1\rightarrow E_2$ is linear for all $x\in X$.

\begin{dfn}\label{defmgr}
Suppose that the Nash manifold $X$ is affine. A linear family $\phi: X\times E_1\rightarrow E_2$ is said to be of  moderate growth if for
every continuous seminorm $\abs{\,\cdot\,}_2$ on $E_2$, there is a positive Nash function $f$ on $X$ and  a continuous seminorm 
$\abs{\,\cdot\,}_1$ on $E_1$ such that
\[
  \abs{\phi(x,u)}_2\leq f(x) \abs{u}_1\quad \textrm{for all } x\in X, u\in E_1.
\]
\end{dfn}

Generally a linear family $\phi: X\times E_1\rightarrow E_2$ is said to be of  moderate growth if there is a finite covering  $\{X_i\}_{i=1}^k $ ($k\geq 0$) of $X$ by affine open Nash submanifolds such that
 $\phi|_{X_i\times E_1}$ is of moderate growth for all $1\leq i\leq k$.

\begin{dfn}\label{deftem}
Suppose that $X$ is affine. A linear family $\phi: X\times E_1\rightarrow E_2$ is said to be tempered if
\begin{itemize}
\item
it is smooth as a map of infinite-dimensional manifolds; and
\item
for every Nash differential operator $D$ on $X$, the linear family
\[
 D\phi: X\times E_1\rightarrow E_2
\]
 is of moderate growth.
 \end{itemize}
\end{dfn}

Generally a linear family $\phi: X\times E_1\rightarrow E_2$ is said to be tempered if there is a finite covering  $\{X_i\}_{i=1}^k $ ($k\geq 0$) of $X$ by affine open Nash submanifolds such that
 $\phi|_{X_i\times E_1}$ is tempered for all $1\leq i\leq k$.

\begin{dfn}\label{tembdm}
Let $X_1$, $X_2$ be Nash manifolds. A map $X_1\times E_1\rightarrow X_2\times E_2$ is called a tempered bundle map if it has the form
\[
  (x, u)\mapsto (f(x), \phi(x, u)),
\]
where $f: X_1\rightarrow X_2$ is a Nash map, and $\phi: X_1\times E_1\rightarrow E_2$ is a tempered linear family.
\end{dfn}

Let $X$ be a Nash manifold and let $\mathsf E$ be a \Fre bundle over $X$, namely, a topological vector bundle over $X$ such that all the fibres are  \Fre spaces. A local chart of $\mathsf E$ is defined to be a triple $(U,E, \phi)$, where $U$ is an open Nash submanifold of $X$, $E$ is a fibre of $\mathsf E$, and
\[
  \phi: U\times E\rightarrow \mathsf E|_U
 \]
is a topological isomorphism of vector bundles over $U$, where $\mathsf E|_U$ denotes the restriction of $\mathsf E$ to $U$.

\begin{dfn}\label{deftemst}
A tempered  structure on $\mathsf E$ is a subset $\CT_\mathsf E$ of the set of all local charts of $\mathsf E$ with the following properties:
\begin{itemize}
\item every two elements  $(U_1,E_1, \phi_1)$, $(U_2, E_2, \phi_2)$ in $\CT_\mathsf E$ are compatible in the sense that
the map
\[
 \phi^{-1}_2\circ \phi_1:  (U_1\cap U_2)\times E_1\rightarrow (U_1\cap U_2)\times E_2
\]
and its inverse are both tempered bundle maps;
\item  for every local chart of $\mathsf E$, if it is compatible with all elements of $\CT_\mathsf E$, then it belongs to $\CT_\mathsf E$;
\item  there exists a finite family $\{(U_i, E_i, \phi_i)\}_{i=1}^k$ ($k\geq 0$) of elements of $\CT_\mathsf E$ such that
     $\{U_i\}_{i=1}^k$ is a covering of $X$.
\end{itemize}
\end{dfn}

\begin{dfn}\label{deftembundle}
A tempered vector bundle is a triple $(X, \mathsf E, \CT_\mathsf E)$, where $X$ is a Nash manifold, $\mathsf E$ is a \Fre bundle over $X$ and $\CT_\mathsf E$ is a tempered structure on $\mathsf E$.
\end{dfn}

When $\CT_\mathsf E$ is understood, we simply call $\mathsf E$ a tempered vector bundle over $X$.
For every tempered vector bundle $(X, \mathsf E, \CT_\mathsf E)$ and  every Nash submanifold $Z$ of $X$,  $\mathsf E|_Z$ is obviously a tempered vector bundle over $Z$.

\begin{dfn}\label{temmor}
Let $(X_1, \mathsf E_1, \CT_{\mathsf E_1})$ and $(X_2, \mathsf E_2, \CT_{\mathsf E_2})$ be two tempered vector bundles.
A map $f: \mathsf E_1\rightarrow \mathsf E_2$ is called a tempered bundle map if there is a Nash map $f_0: X_1\rightarrow X_2$ such that the diagram
\[
 \begin{CD}
           \mathsf E_1@>f >> \mathsf E_2 \\
            @V VV           @VV V\\
           X_1@>f_0>> X_2\\
  \end{CD}
\]
commutes, and for every $(U_1, E_1, \phi_1)\in \CT_{\mathsf E_1}$ and $(U_2, E_2, \phi_2)\in \CT_{\mathsf E_2}$ with $f_0(U_1)\subset U_2$, the map
\[
\phi_2^{-1} \circ  f\circ \phi_1: U_1\times E_1\rightarrow U_2\times E_2
\]
is a tempered bundle map in the sense of Definition \ref{tembdm}.
\end{dfn}

\subsection{Schwartz sections}

For the definition of Schwartz functions on Nash manifolds, we refer the reader to \cite{AG1}.
Let $(X, \mathsf E, \CT_{\mathsf E})$ be a tempered vector bundle. Suppose that  $\{(U_i, E_i, \phi_i)\}_{i=1}^k$ ($k\geq 0$) are elements of $\CT_{\mathsf E}$ such that $\{U_i\}_{i=1}^{k}$ is a covering of $X$. Write $\Gamma^{\varsigma}(U_{i},U_{i}\times E_{i})$ for the space of the sections which correspond to Schwartz functions in
$\CS(U_i, E_i)$. This is obviously a \Fre space.
Define
\[
  \Gamma^{\varsigma}(U_{i},\mathsf E|_{U_i}):=\{\phi_{i}\circ s\, :\,  s\in\Gamma^{\varsigma}(U_{i},U_{i}\times E_{i})\},
  \]
  which is also obviously  a \Fre space.

 Denote by
$\Gamma(X, \mathsf E)$ the space of continuous sections of the bundle $\mathsf E$ over $X$. Then extension by zero gives a continuous linear map
\[
\bigoplus_{i=1}^k\Gamma^{\varsigma}(U_{i},\mathsf E|_{U_i})\rightarrow\Gamma(X,\mathsf E).
\]

\begin{dfn}
With the notation as above, the Schwartz sections $\Gamma^{\varsigma}(X, \mathsf E)$ of the tempered vector bundle $\mathsf E$ over the Nash manifold $X$ is defined to be the image of the above map, equipped with the quotient topology of the domain.
\end{dfn}

The definition of  $\Gamma^{\varsigma}(X, \mathsf E)$ (as a topological vector space) does not depend on the choice of the local charts $\{(U_i, E_i, \phi_i)\}_{i=1}^k$, see \cite[Proposition 6.2]{CS}.

Let $X$ be a $G$-Nash manifold, namely, a Nash manifold carrying a Nash action $G\times X\rightarrow X$. By a tempered $G$-vector bundle over $X$, we mean a tempered vector bundle
$\mathsf E$ over $X$, together with an action $G\times \mathsf E\rightarrow \mathsf E$ which is a tempered bundle map. Here $G\times \mathsf E$ is obviously viewed as a tempered vector bundle over $G\times X$.

\begin{prpd}[\cite{CS}, Proposition 6.3]
Suppose that the Nash manifold $X$ carries a Nash $G$-action. Let $\mathsf E$ be a tempered $G$-vector bundle over $X$. Then for every $g\in G$ and $\phi\in \Gamma^{\varsigma}(X, \mathsf E)$,
\[
    \begin{array}{rcl}
g.\phi:  X&\rightarrow &\mathsf E,\\
 x&\mapsto & g.(\phi(g^{-1}.x))
 \end{array}
\]
is a section in $\Gamma^{\varsigma}(X, \mathsf E)$. Moreover,  the \Fre space $\Gamma^{\varsigma}(X, \mathsf E)$  is a representation in $\CS\mathrm{mod}_G$ under the action
\[
  (g, \phi)\mapsto g.\phi.
\]
\end{prpd}

\section{Proof of the main results}

We prove Theorem \ref{main} by induction on the number of $G$-orbits in $X$.
Firstly, assume that $G$ acts on $X$ transitively. By \cite[Proposition 6.7]{CS}, 
$$\Gamma^{\varsigma}(X,\SE)\cong\ind_{G_x}^{G}\SE_x$$ 
as representations of $G$, for any $x\in X$. For every $i\in\BZ$, by Theorem \ref{shapiro},
\[
\oH_{i}^{\CS}(G;\Gamma^{\varsigma}(X,\SE))\cong\oH_{i}^{\CS}(G;\ind_{G_x}^{G}\SE_x)\cong\oH_{i}^{\CS}(G_x;\SE_x\otimes\delta_{G/G_x}),
\]
which is finite dimensional by the first assumption in Theorem \ref{main}. Thus $\Gamma^{\varsigma}(X,\SE)$ is homologically finite in this case.

Now assume that the number $n$ of $G$-orbits in $X$ is larger than 1 and Theorem \ref{main} holds for any tempered $G$-vector bundle $(X^\prime,\SE^\prime)$ (satisfying the conditions in Theorem \ref{main}) with the number of $G$-orbits in $X^\prime$ less than $n$. Let $Z$ be a $G$-orbit in $X$ with minimal dimension among all the orbits. Then by \cite[Proposition 3.6]{Su}, $Z$ must be closed in $X$. 
Denote by $U:=X\setminus Z$, which is a $G$-invariant open Nash submanifold of $X$. Clearly, $(U,\SE|_U)$ is a tempered $G$-vector bundle satisfying the conditions in Theorem \ref{main} and the number of $G$-orbits in $U$ is less than $n$. By the induction assumption, $\Gamma^{\varsigma}(U,\SE|_U)$ is homologically finite. By \cite[Theorem 5.4.1]{AG1}, the extension by zero yields a closed linear embedding
\[
\Gamma^{\varsigma}(U,\SE|_U)\hookrightarrow\Gamma^{\varsigma}(X,\SE),
\]
and we identify $\Gamma^{\varsigma}(U,\SE|_U)$ with its image in $\Gamma^{\varsigma}(X,\SE)$. Define
\[
\Gamma_{Z}^{\varsigma}(X, \mathsf E):=\Gamma^{\varsigma}(X, \mathsf E)/\Gamma^{\varsigma}(U, \mathsf E|_U).
\]

\begin{lem}\label{shor}
Let $0\rightarrow V_1\rightarrow V_2 \rightarrow V_3\rightarrow 0$ be a short exact sequence in the category $\CS\mathrm{mod}_G$.
If any two of $V_1, V_2$ and $V_3$ are homologically finite, then so is the third one.
\end{lem}

\begin{proof}
By \cite[Corollary 7.8]{CS}, the short exact sequence $0\rightarrow V_1\rightarrow V_2 \rightarrow V_3\rightarrow 0$ yields a long exact sequence
\[
\cdots \rightarrow\oH^{\CS}_{i+1}(G; V_{3})\rightarrow\oH^{\CS}_{i}(G; V_{1})\rightarrow\oH^{\CS}_{i}(G; V_{2})\rightarrow \oH^{\CS}_{i}(G; V_{3})\rightarrow\cdots
\]
of Schwartz homologies. Now the lemma follows directly from the definition of homological finiteness.
\end{proof}

From the short exact sequence
\[
0\rightarrow \Gamma^{\varsigma}(U,\SE|_U)\rightarrow \Gamma^{\varsigma}(X,\SE) \rightarrow \Gamma_{Z}^{\varsigma}(X, \mathsf E)\rightarrow 0
\]
and Lemma \ref{shor}, to complete the proof it is enough to show that $\Gamma_{Z}^{\varsigma}(X, \mathsf E)$ is homologically finite.
For this quotient space, we have the following characterization.

\begin{lem}[Borel's Lemma]\label{borel}
$\Gamma_{Z}^{\varsigma}(X, \mathsf E)$ has a natural countable decreasing filtration $\{\Gamma_{Z}^{\varsigma}(X, \mathsf E)_k\}_{k\geq0}$ by closed subspaces such that
\begin{itemize}
\item[(a)] The natural map  $$\Gamma^{\varsigma}_Z(X, \mathsf E)\rightarrow \varprojlim_k \, \Gamma_Z^{\varsigma}(X, \mathsf E)/\Gamma_{Z}^{\varsigma}(X, \mathsf E)_k$$
is a topological linear  isomorphism.
\item[(b)] For all integers $k\geq0$, there is naturally a topological linear isomorphism
\[
  \Gamma_{Z}^{\varsigma}(X, \mathsf E)_k/\Gamma_{Z}^{\varsigma}(X, \mathsf E)_{k+1}\cong \Gamma^\varsigma(Z, \Sym^k(\mathrm N^*_Z(X))\otimes \mathsf E|_Z),
\]
where $\Gamma_{Z}^{\varsigma}(X, \mathsf E)_0=\Gamma_{Z}^{\varsigma}(X, \mathsf E)$.
\end{itemize}
\end{lem}

Here 
\[
  \mathrm{N}_{Z}(X):=\bigsqcup_{x\in Z} \frac{ \mathrm T_x(X)}{\mathrm T_x(Z)}\otimes_\BR \BC
\]
denotes the complexified normal bundle of $Z$ in $X$ and $\mathrm{N}^*_{Z}(X)$ its dual bundle, which is called the complexified conormal bundle. In our situation, the closed subspaces $\Gamma_{Z}^{\varsigma}(X, \mathsf E)_k$ are naturally $G$-invariant and the isomorphisms in Lemma \ref{borel} are $G$-equivariant. For a proof of Lemma \ref{borel}, see for example \cite[Propositions 8.2 and 8.3]{CS}.

\begin{lem}\label{factor}
$\Gamma_Z^{\varsigma}(X, \mathsf E)/\Gamma_{Z}^{\varsigma}(X, \mathsf E)_k$ is homologically finite for every integer $k\geq 0$.
\end{lem}

\begin{proof}
By \cite[Proposition 6.7]{CS} and Theorem \ref{shapiro},
\[
 \begin{array}{rcl}
\oH_i^{\CS}(G;\Gamma^{\varsigma}(Z,\Sym^j(\mathrm{N}^*_{Z}(X))\otimes\SE|_Z))&\cong&\oH_i^{\CS}(G;\ind_{G_z}^G(\Sym^j(\mathrm{N}^*_{z})\otimes\SE_z))\\
 &\cong& \oH_{i}^{\CS}(G_z;\Sym^j(\mathrm{N}^*_{z})\otimes\SE_z\otimes\delta_{G/G_z}),
 \end{array}
\]
which by the first assumption of Theorem \ref{main}, is finite dimensional for every $i\in\BZ$ and $j\geq0$, here $z\in Z$. Thus 
$\Gamma^{\varsigma}(Z,\Sym^j(\mathrm{N}^*_{Z}(X))\otimes\SE|_Z)$ is homologically finite, for every $j\geq0$.
By Lemma \ref{borel}, we conclude that $\Gamma_Z^{\varsigma}(X, \mathsf E)_j/\Gamma_{Z}^{\varsigma}(X, \mathsf E)_{j+1}$
is homologically finite, for every $j\geq0$. Now the lemma follows from Lemma \ref{shor}.
\end{proof}

\begin{lem}\label{comu}
The natural map
\[
\oH_{i}^\CS(G;\Gamma_{Z}^{\varsigma}(X, \mathsf E))\rightarrow\varprojlim_{k}\oH_{i}^\CS(G;\Gamma_{Z}^{\varsigma}(X, \mathsf E)/\Gamma_{Z}^{\varsigma}(X, \mathsf E)_k)
\]
is a linear isomorphism, for every $i\in\BZ$.
\end{lem}

\begin{proof}
By Lemma \ref{factor}, $\oH_{i+1}^\CS(G;\Gamma_{Z}^{\varsigma}(X, \mathsf E)/\Gamma_{Z}^{\varsigma}(X, \mathsf E)_k)$ is finite dimensional, for every $k\geq0$. Now the lemma follows directly from \cite[Lemma 8.4]{CS}.
\end{proof}

We need to show that $\oH_{i}^\CS(G;\Gamma_{Z}^{\varsigma}(X, \mathsf E))$ is finite dimensional for every $i\in\BZ$.

\begin{lem}\label{fine}
Let $i\in\BZ$, the natural map
\[
\oH_{i}^\CS(G;\Gamma_{Z}^{\varsigma}(X, \mathsf E)/\Gamma_{Z}^{\varsigma}(X, \mathsf E)_{k+1})\rightarrow\oH_{i}^\CS(G;\Gamma_{Z}^{\varsigma}(X, \mathsf E)/\Gamma_{Z}^{\varsigma}(X, \mathsf E)_k)
\]
induced from the surjection $$\Gamma_{Z}^{\varsigma}(X, \mathsf E)/\Gamma_{Z}^{\varsigma}(X, \mathsf E)_{k+1}\rightarrow\Gamma_{Z}^{\varsigma}(X, \mathsf E)/\Gamma_{Z}^{\varsigma}(X, \mathsf E)_{k}$$ is a linear isomorphism for sufficiently large $k$.
\end{lem}

\begin{proof}
By \cite[Corollary 7.8]{CS}, the short exact sequence
\[
0\rightarrow \Gamma_Z^{\varsigma}(X, \mathsf E)_k/\Gamma_{Z}^{\varsigma}(X, \mathsf E)_{k+1}\rightarrow \Gamma_Z^{\varsigma}(X, \mathsf E)/\Gamma_{Z}^{\varsigma}(X, \mathsf E)_{k+1} \rightarrow \Gamma_Z^{\varsigma}(X, \mathsf E)/\Gamma_{Z}^{\varsigma}(X, \mathsf E)_{k}\rightarrow 0
\]
in the category $\CS\mathrm{mod}_G$ yields a long exact sequence
\[
 \begin{array}{rcl}
&\rightarrow&\oH_i^{\CS}(G;\Gamma_Z^{\varsigma}(X, \mathsf E)_k/\Gamma_{Z}^{\varsigma}(X, \mathsf E)_{k+1})
\rightarrow\oH_i^{\CS}(G;\Gamma_Z^{\varsigma}(X, \mathsf E)/\Gamma_{Z}^{\varsigma}(X, \mathsf E)_{k+1})\\
&\rightarrow&\oH_i^{\CS}(G;\Gamma_Z^{\varsigma}(X, \mathsf E)/\Gamma_{Z}^{\varsigma}(X, \mathsf E)_{k})
\rightarrow\oH_{i-1}^{\CS}(G;\Gamma_Z^{\varsigma}(X, \mathsf E)_k/\Gamma_{Z}^{\varsigma}(X, \mathsf E)_{k+1})\rightarrow
 \end{array}
\]
of Schwartz homologies.
By Lemma \ref{borel} and Theorem \ref{shapiro},
\[
 \begin{array}{rcl}
 \oH_i^\CS(G;\Gamma_Z^{\varsigma}(X, \mathsf E)_k/\Gamma_{Z}^{\varsigma}(X, \mathsf E)_{k+1})&\cong&
\oH_i^{\CS}(G;\Gamma^{\varsigma}(Z,\Sym^k(\mathrm{N}^*_{Z}(X))\otimes\SE|_Z))\\
&\cong&\oH_i^{\CS}(G;\ind_{G_z}^G(\Sym^k(\mathrm{N}^*_{z})\otimes\SE_z))\\
 &\cong& \oH_{i}^{\CS}(G_z;\Sym^k(\mathrm{N}^*_{z})\otimes\SE_z\otimes\delta_{G/G_z}),
 \end{array}
\]
which vanishes for sufficiently large $k$, by the second assumption in Theorem \ref{main}. Similarly,
$\oH_{i-1}^{\CS}(G;\Gamma_Z^{\varsigma}(X, \mathsf E)_k/\Gamma_{Z}^{\varsigma}(X, \mathsf E)_{k+1})$ vanishes for sufficiently larke $k$. Thus the natural map
\[
\oH_{i}^\CS(G;\Gamma_{Z}^{\varsigma}(X, \mathsf E)/\Gamma_{Z}^{\varsigma}(X, \mathsf E)_{k+1})\rightarrow\oH_{i}^\CS(G;\Gamma_{Z}^{\varsigma}(X, \mathsf E)/\Gamma_{Z}^{\varsigma}(X, \mathsf E)_k)
\]
is a linear isomorphism for sufficiently large $k$.
\end{proof}

By Lemmas \ref{factor}, \ref{comu} and \ref{fine}, we conclude that $\oH_{i}^\CS(G;\Gamma_{Z}^{\varsigma}(X, \mathsf E))$ is finite dimensional for every $i\in\BZ$, i.e., $\Gamma_{Z}^{\varsigma}(X, \mathsf E)$ is homologically finite, which completes the proof of Theroem \ref{main}.

Now we prove Theorem \ref{main2}. Let $x\in X$.
Since all the fibres of $\SE$ are finite dimensional, the representation $\mathsf E_x\otimes\Sym^k(\mathrm{N}_{x}^*)\otimes\delta_{G/G_x}$ of $G_x$ is finite dimensional, for every $k\geq0$. By \cite[Theorem 7.7]{CS},  it must be homologically finite. Since the trivial representation of $G_x$ does not occur as a subquotient of $\mathsf E_x\otimes\Sym^k(\mathrm{N}_{x}^*)\otimes\delta_{G/G_x}$ for sufficiently large $k$, it follows from \cite[Proposition 7.11 and Corollary 7.8]{CS} that $\oH_{i}^{\CS}(G_x;\Sym^k(\mathrm{N}^*_{x})\otimes\SE_x\otimes\delta_{G/G_x})$ vanishes for every $i\in\BZ$ and sufficiently large $k$. Now Theorem \ref{main2} follows from Theorem \ref{main}.


\begin{thebibliography}{99}

\bibitem{AG1}
A. Aizenbud and D. Gourevitch, \textit{Schwartz functions on Nash manifolds}, Int. Math. Res. Not. 2008, no.5, Art. ID rnm155, https://doi.org/10.1093/imrn/rnm155.



\bibitem{AGKL}
A. Aizenbud, D. Gourevitch, B. Kroetz and G. Liu, \textit{Hausdorffness for Lie algebra homology of Schwartz spaces and applications to the comparison conjecture}, Math. Zeit. \textbf{283}, 979--992 (2016), https://doi.org/10.1007/s00209-016-1629-6. 





\bibitem{BW}
P. Blanc and D. Wigner, \textit{Homology of Lie groups and Poincar\'{e} duality},  Letters in Math. Phy. \textbf{7}, 259--270 (1983), https://doi.org/10.1007/bf00400442.


\bibitem{BoW}
A. Borel and N. Wallach, \textit{Continuous cohomology, discrete subgroups, and representations of reductive groups}, Annals of Mathematics Studies \textbf{94}, 1980.


\bibitem{CS}
Y. Chen and B. Sun, \textit{Schwartz homologies of representations of almost linear Nash groups}, J. Funct. Anal. (2020), https://doi.org/10.1016/j.jfa.2020.108817.

\bibitem{CW}
W. Casselman and D. Wigner, \textit{Continuous cohomology and a conjecture of Serre's}, Inventiones Math. \textbf{25}, 199--211 (1974), https://doi.org/10.1007/bf01389727.



\bibitem{Fd}
F. du Cloux, \textit{Sur les repr\'{e}sentations diff\'{e}rentiables des groupes de Lie alg\'{e}briques},  Ann. Sci. Ecole Norm. Sup. \textbf{24}, no. 3, 257--318 (1991), https://doi.org/10.24033/asens.1628.


\bibitem{LS}
Y. Liu and B. Sun, \textit{Uniqueness of Fourier--Jacobi Models: the Archimedean case}, J. Funct. Anal. \textbf{265}, 3325--3344 (2013), https://doi.org/10.1016/j.jfa.2013.08.034.


\bibitem{Sh}
M. Shiota, \textit{Nash Manifolds}, Lect. Notes Math., vol. \textbf{1269}, Springer-Verlag, 1987.

\bibitem{Sh2}
M. Shiota, \textit{Nash Functions and Manifolds}, in Lectures in Real Geometry, F. Broglia (ed.), W. de Gruyter, Berlin, New York, 1996.

\bibitem{Su}
B. Sun, \textit{Almost linear Nash groups}, China. Ann. Math. Ser. B \textbf{36}, 355--400 (2015), https://doi.org/10.1007/s11401-015-0915-7.




\end{thebibliography}
\end{document}